\documentclass[a4paper, 12pt]{amsart}

\usepackage[all]{xy}
\usepackage{amssymb}

\newcount\hours
\newcount\minutes       
\def\timeofday{
\hours=\time
\minutes=\hours
\divide\hours by60
\multiply\hours by60
\advance\minutes by-\hours
\divide\hours by60
\ifnum\hours>9\else0\fi\the\hours:\ifnum\minutes>9\else
0\fi\the\minutes}
\def\predate{\date{\the\day\ \ifcase\month\or
  January\or February\or March\or April\or May\or June\or July\or
        August\or September\or October\or November\or
           December\fi\ \the\year\ --- \timeofday\ --- Preliminary
                  Version}}

\def\<{\langle}
\def\>{\rangle}
\let\ipscriptstyle=\scriptscriptstyle
\def\lipsqueeze{{\mskip -3.0mu}}
\def\ripsqueeze{{\mskip -3.0mu}}
\def\ipcomma{\nobreak\mathrel{,}\nobreak}
\newbox\ipstrutbox
\setbox\ipstrutbox=\hbox{\vrule height8.5pt
width 0pt}
\def\ipstrut{\copy\ipstrutbox}
\def\lip#1<#2,#3>{\mathopen{\relax_{\ipstrut\ipscriptstyle{
#1}}\lipsqueeze
\langle} #2\ipcomma #3 \rangle}
\def\blip#1<#2,#3>{\mathopen{\relax_{\ipstrut
\ipscriptstyle{ #1}}\lipsqueeze\bigl\langle} #2\ipcomma #3 \bigr\rangle}
\def\rip#1<#2,#3>{\langle #2\ipcomma #3
\rangle_{\ripsqueeze\ipstrut\ipscriptstyle{#1}}}
\def\brip#1<#2,#3>{\bigl\langle #2\ipcomma #3
\bigr\rangle_{\ripsqueeze\ipstrut\ipscriptstyle{#1}}}
\def\angsqueeze{\mskip -6mu}
\def\smangsqueeze{\mskip -3.7mu}
\def\trip#1<#2,#3>{\langle\smangsqueeze\langle #2\ipcomma #3
\rangle\smangsqueeze\rangle_{\ripsqueeze\ipstrut\ipscriptstyle{#1}}}
\def\btrip#1<#2,#3>{\bigl\langle\angsqueeze\bigl\langle #2\ipcomma
#3
\bigr\rangle
\angsqueeze\bigr\rangle_{\ripsqueeze\ipstrut\ipscriptstyle{#1}}}
\def\tlip#1<#2,#3>{\mathopen{\relax_{\ipstrut\ipscriptstyle{
#1}}\lipsqueeze \langle\smangsqueeze\langle} #2\ipcomma #3
\rangle\smangsqueeze\rangle}
\def\btlip#1<#2,#3>{\mathopen{\relax_{\ipstrut\ipscriptstyle{
#1}}\lipsqueeze
\bigl\langle\angsqueeze\bigl\langle} #2\ipcomma #3
\bigr\rangle\angsqueeze\bigr\rangle}

\def\ip(#1|#2){(#1\mid #2)}
\def\bip(#1|#2){\bigl(#1 \mid #2\bigr)}
\def\Bip(#1|#2){\Bigl( #1 \bigm| #2 \Bigr)}

\newtheorem{theorem}{Theorem}[section]
\newtheorem{thm}[theorem]{Theorem}
\newtheorem{lemma}[theorem]{Lemma}
\newtheorem{prop}[theorem]{Proposition}
\newtheorem{cor}[theorem]{Corollary}

\theoremstyle{remark}

\newtheorem{remark}[theorem]{Remark}

\newtheorem*{problem*}{Problem}
\newtheorem*{remark*}{Remark}
\newtheorem*{convention*}{Convention}
\newtheorem*{notation*}{Notation}
\newtheorem*{examples*}{Examples}
\newtheorem*{example*}{Example}
\newtheorem*{warning*}{Warning}

\numberwithin{equation}{section}

\def\R{{\mathbb R}}

\def\Z{{\mathbb Z}}

\def\K{{\mathcal K}}

\newcommand{\FF}{\operatorname{\mathtt{F}}}
\newcommand{\Fix}{\operatorname{\mathtt{Fix}}}
\newcommand{\RCP}{\operatorname{\mathtt{RCP}}}

\newcommand{\CC}{\operatorname{\mathtt{C*}}}

\newcommand{\Aa}{\operatorname{\mathtt{C*act}}}

\newcommand\CT{\operatorname{\mathtt{CT}}}

\newcommand\tensor{\otimes}

\newcommand{\rt}{\textup{rt}}
\newcommand{\id}{\operatorname{id}}

\newcommand{\Ind}{\operatorname{Ind}}
\newcommand{\Aut}{\operatorname{Aut}}

\newcommand{\newspan}{\operatorname{span}}
\newcommand{\spn}{\newspan}

\def\Res{\operatorname{\mathtt{Res}}}

\newcommand\dec{\operatorname{dec}}

\newcommand\bbetatrt{\overline{\rt\tensor\beta}}
\newcommand\balphatrt{\overline{\rt\otimes\alpha}}

\begin{document}

\title[Naturality]{Naturality of Symmetric Imprimitivity Theorems}
\author[an Huef]{Astrid an Huef}
\address{Department of Mathematics and Statistics\\ University of
  Otago\\ PO Box 56\\ Dunedin 9054 \\ New Zealand}
\email{astrid@maths.otago.ac.nz}

\author[Kaliszewski]{S. Kaliszewski}
\address{School of Mathematical and Statistical Sciences\\Arizona 
State University\\Tempe\\ AZ
85287-1804\\USA} \email{kaliszewski@asu.edu}

\author[Raeburn]{Iain Raeburn}
\address{Department of Mathematics and Statistics\\ University of
  Otago\\ PO Box 56\\ Dunedin 9054 \\ New Zealand}
\email{iraeburn@maths.otago.ac.nz}

\author[Williams]{Dana P. Williams}
\address{Department of Mathematics\\Dartmouth College\\ 
Hanover, NH 03755\\USA}
\email{dana.williams@dartmouth.edu}

\begin{abstract}  The first imprimitivity theorems identified the representations of groups or dynamical systems which are induced from representations of a subgroup.  Symmetric imprimitivity theorems identify pairs of crossed products by different groups which are Morita equivalent, and hence have the same representation theory.  Here we consider  commuting actions of groups $H$ and $K$ on a $C^*$-algebra which are saturated and proper as defined by Rieffel in 1990.   Our main result says that the resulting Morita equivalence of crossed products is natural in the sense that it is compatible with homomorphisms and induction processes.
\end{abstract}


\date{25 March 2011}

\thanks{This research was supported by the University of Otago and the
  Edward Shapiro fund at Dartmouth College.}

\maketitle

\section{Introduction}
\label{sec:introduction}

Suppose that a locally compact group~$G$ acts freely and properly on
the right of a locally compact space~$T$, and $\rt$ is the induced
action on $C_0(T)$. Green proved in \cite{G} that the crossed product
$C_{0}(T)\rtimes_{\rt}G$ is Morita equivalent to $C_{0}(T/G)$.
Raeburn and Williams considered diagonal actions $\rt\otimes\alpha$ on
$G$ on $C_0(T, B)= C_{0}(T)\tensor B$, and showed in \cite{RW} that
$C_{0}(T,B)\rtimes_{\rt\otimes\alpha}G$ is Morita equivalent to the
induced algebra $\Ind_{G}^{T}(B,\alpha)$.  Motivated by the idea that
$C_{0}(T/G)$ and $\Ind_{G}^{T}(B,\alpha)$ are playing the role of a
fixed-point algebra for the action of $G$, Rieffel studied a family of
proper actions $(A,\alpha)$ for which there is a generalized
fixed-point algebra $A^\alpha$ in $M(A)$ \cite{proper}.  Rieffel
proved in particular that if $\alpha$ is an action of $G$ on a
$C^*$-algebra $A$ and if $\phi\colon C_0(T)\to M(A)$ is an equivariant
nondegenerate homomorphism, then the reduced crossed product
$A\rtimes_{\alpha,r}G$ is Morita equivalent to $A^\alpha$ (see
\cite[Theorem~5.7]{integrable} and \cite[Corollary~1.7]{proper});
taking $A=C_{0}(T,B)$ gives the result in \cite{RW}.  

Rieffel's construction of $A^\alpha$ starts from a dense subalgebra $A_0$ of $A$ with properties like those of $C_c(T)$ in $C_0(T)$; while in practice there always seems to be an obvious candidate for $A_0$ which gives the ``right answer'', $A^\alpha$ does ostensibly depend on the choice of $A_0$, and several authors have tried in vain to find a canonical choice \cite{exel,integrable}. Alternatively, as in \cite{kqrproper}, we can try to say upfront what ``right'' means, and prove that  Rieffel's construction has these properties. The idea is, loosely, that constructions should be functorial and isomorphisms, such as those implemented by Morita equivalences, should be natural. This program has already had some significant applications, especially in nonabelian duality for crossed products of $C^*$-algebras \cite{kqrproper,aHKRW-part2}.

In
\cite{aHKRW-trans}, we showed that the assigments
$(A,\alpha,\phi)\mapsto A\rtimes_{\alpha,r}G$ and
$(A,\alpha,\phi)\mapsto A^\alpha$ can be extended to functors $\RCP$
and $\Fix$ between certain categories whose morphisms are derived from
right-Hilbert bimodules, and whose isomorphisms are given by Morita
equivalences.  The main result of \cite{aHKRW-trans} says that
Rieffel's Morita equivalences give a natural isomorphism between
$\RCP$ and $\Fix$.

All these Morita equivalences have symmetric versions.  For Green's
theorem, the symmetric version involves commuting free and proper
actions of two groups $H$ and $K$ on the same space $T$, and says that
$C_0(T/H)\rtimes K$ is Morita equivalent to $C_0(T/K)\rtimes H$ (this
is marginally more general than the version proved in
\cite{rie:pspm82}). If in addition $\alpha$ and $\beta$ are commuting
actions of $H$ and $K$ on a $C^{*}$-algebra $B$, then the symmetric
imprimitivity theorem of \cite{kas, rae} gives a Morita equivalence
between crossed products $\Ind_{H}^{T}(B,\alpha)\rtimes_{\bbetatrt} K$
and $\Ind_{K}^{T}(B,\beta)\rtimes_{\balphatrt} H$.  Work of Quigg and
Spielberg \cite{qs} implies that there is a similar Morita equivalence
for the reduced crossed products.  In
\cite[Corollary~3.8]{aHRWproper2}, we found a symmetric version of
Rieffel's equivalence for a pair of commuting proper actions on a
$C^*$-algebra $A$, and then recovered the Quigg-Spielberg theorem by
taking $A=C_0(T,B)$ (see \cite[\S4]{aHRWproper2}). Symmetric
imprimitivity theorems have found significant applications (see
\cite{KW, CEOO}, for example), and there are analogues for groupoids
\cite{mrw, ren}, for graph algebras \cite{pr}, and for Fell bundles
\cite{MW}.
  
Here we consider commuting free and proper actions of $H$ and $K$ on
$T$, and form a category whose objects $(A,\sigma,\tau,\phi)$ consist
of commuting actions $\sigma: H\to \Aut A$ and $\tau: K\to \Aut A$,
and an equivariant nondegenerate homomorphism $\phi: C_0(T)\to M(A)$.
We show that there are functors $\Fix_H$ and $\Fix_K$ based on the
assignments $(A,\sigma,\tau,\phi)\mapsto (A^\sigma,\bar\tau)$ and
$(A,\sigma,\tau,\phi)\mapsto (A^\tau,\bar\sigma)$, and prove that
Rieffel's bimodules $X(A,\sigma,\tau,\phi)$
from~\cite[Corollary~3.8]{aHRWproper2} give a natural isomorphism
between the functors $\RCP\circ\Fix_H$ and $\RCP\circ\Fix_K$.  This is
interesting even in the situation of \cite{rae}, where it gives the
naturality of Quigg and Spielberg's symmetric imprimitivity theorem,
and in the one-sided case, where it yields naturality of the Morita
equivalence of \cite{RW}.

\section{Preliminaries}
\label{prelim}

Throughout this paper, $G$, $H$, and $K$ are locally compact groups,
all of which act on the right of a locally compact space $T$. The
actions of $H$ and $K$ are always assumed to commute. We denote by
$\rt$ the action of any of them on $C_0(T)$ by right translation:
$\rt_g(f)(t) = f(t\cdot g)$

We use the same categories and notation as in~\cite{aHKRW-trans}.  In
particular, $\CC$ is the category whose objects are $C^*$-algebras and
whose morphisms are isomorphism classes of right-Hilbert bimodules. In
the category $\Aa(G)$, the objects $(A,\alpha)$ consist of an action
$\alpha$ of $G$ on a $C^*$-algebra $A$, and the morphisms $[X,u]:
(A,\alpha)\to (B,\beta)$ are isomorphism classes of right-Hilbert
$A$--$B$ bimodules $X$ with an $\alpha$--$\beta$ compatible action~$u$
of $G$.

In the semi-comma category $\Aa(G,(C_0(T),\rt))$, the objects are
triples $(A,\alpha,\phi)$, where $(A,\alpha)$ is an object in $\Aa(G)$
and $\phi: C_0(T)\to M(A)$ is an $\rt$--$\alpha$ equivariant
nondegenerate homomorphism.  The morphisms from $(A,\alpha,\phi)$ to
$(B,\beta,\psi)$ are the same as the morphisms from $(A,\alpha)$ to
$(B,\beta)$ in $\Aa(G)$.\footnote{In \cite{aHKRW-trans}, we were only
  interested in free and proper actions of $G$, but there is no reason
  to make this restriction when defining the semi-comma category
  $\Aa(G,(C_{0}(T),\rt))$.}
 
Commuting actions $\sigma$, $\tau$ of $H$, $K$ (on spaces,
$C^*$-algebras or Hilbert modules) are essentially the same as actions
$\sigma\times\tau$ of $G=H\times K$, and hence we can apply the
construction of the previous paragraph with $G=H\times K$. However, to
emphasize the symmetry of our situation, we view the comma category as
a category $\Aa(H,K,(C_0(T),\rt))$ in which the objects are quadruples
$(A,\sigma,\tau,\phi)$ such that $(A,\sigma\times\tau,\phi)$ is an
object of the semi-comma category $\Aa(H\times K,(C_0(T),\rt))$; the
morphisms from $(A,\sigma,\tau,\phi)$ to $(B,\mu,\nu,\psi)$ are then
triples $[X,u,v]$ such that $[X,u\times v]$ is a morphism from
$(A,\sigma\times\tau,\phi)$ to $(B,\mu\times\nu,\psi)$ in $\Aa(H\times
K,(C_0(T),\rt))$.

\begin{remark}
  We stress that, when we assume that $H$ and $K$ act freely and
  properly on $T$, we are not assuming that $H\times K$ acts freely
  and properly, because then we would lose the main applications of
  the symmetric imprimitivity theorem. For example, let $T=\R$, $H=\Z$
  and $K=\Z$, take an irrational number $\theta$, and define $r\cdot
  h:= r+h\theta$ and $r\cdot k=r+k$ for $r\in T$, $h\in H$, $k\in K$:
  then the symmetric imprimitivity theorem implies that the irrational
  rotation algebra $A_\theta$ is Morita equivalent to
  $A_{\theta^{-1}}$.
\end{remark}

\section{Defining the Functors}
\label{sec:finding-functors}

In this section we define functors on $\Aa(H,K,(C_0(T),\rt))$
analogous to $\RCP$ and $\Fix$ from~\cite{aHKRW-trans}, but which only
deal with the $H$- or $K$-part of the action, and which take values in
an equivariant category.

\begin{prop}
  \label{prop-rcph-functor}
  The assignments
  \begin{equation*}
    (A,\sigma,\tau,\phi)\mapsto(A\rtimes_{\sigma,r}H,\tau\rtimes\id_{H})
    \quad\text{and}\quad 
    [X,u,v]\mapsto [X\rtimes_{u,r}H,v\rtimes\id_{H}] 
  \end{equation*}
  define a functor $\RCP_{H}$ from $\Aa(H,K,(C_{0}(T),\rt))$ to
  $\Aa(K)$, where $\tau\rtimes\id_H$ is the action of $K$ given by
  $(\tau\rtimes\id_H)_k(f)(h) = \tau_k(f(h))$ for $f\in C_c(H,A)$
  \textup(and similarly for $v\rtimes\id_H$\textup).
\end{prop}

\begin{proof}
  Let $\FF: \Aa(H\times K)\to\Aa(H\times K)$ be the
  subgroup-crossed-product functor from
  \cite[Theorem~3.24]{enchilada}, applied to the normal subgroup
  $H\subset H\times K$.  With $\alpha^{\dec}$ and $w^{\dec}$ the
  ``decomposition actions'' of $H\times K$ as defined in
  \cite[Section~3.3.1]{enchilada}, $\FF$ is given by
  \begin{equation*}
    (A,\alpha)\mapsto
    (A\rtimes_{\alpha|,r}H,\alpha^{\dec}) \quad\text{and}
    \quad [X,w]\mapsto [X\rtimes_{w|,r}H,w^{\dec}].
  \end{equation*}
  With $\alpha=\sigma\times\tau$, for $k\in K$ and $f\in C_{c}(H,A)$
  we have
  \[
  \alpha^{\dec}_{(e,k)}(f)(h) = \Delta_{H}(e) \sigma_{e} \bigl(
  \tau_{k}\bigl( f(e^{-1}he)\bigr)\bigr) = \tau_k(f(h)) =
  (\tau\rtimes\id_H)_k(f)(h),
  \]
  and similarly, if $w=u\times v$, then $w^{\dec}_{(e,k)} =
  (v\rtimes\id_H)_k$.  So if we let $\Res: \Aa(H\times K)\to\Aa(K)$ be
  the restriction functor from \cite[Corollary~3.17]{enchilada}
  applied to the subgroup $K\subset H\times K$, we see that the
  composition $\Res\circ\FF$ is given by
  \[
  (A,\sigma\times\tau)\mapsto (A\rtimes_{\sigma,r}H,\tau\rtimes\id_H)
  \quad\text{and} \quad [X,u\times v]\mapsto
  [X\rtimes_{u,r}H,v\rtimes\id_H].
  \]
  Defining $\RCP_H$ to be the composition of $\Res\circ\FF$ with the
  forgetful functor from $\Aa(H\times K,(C_0(T),\rt))$ to $\Aa(H\times
  K)$ which takes $(A,\alpha,\phi)$ to $(A,\alpha)$ gives the result.
\end{proof}

For the next result, we observe that the proof of
\cite[Proposition~4.1]{aHKRW-part2} only requires that the normal
subgroup $N\subset G$ acts freely and properly; the basic
factorization result in \cite[Corollary~2.3]{aHKRW-trans}, which is
invoked in the proof, applies in the semi-comma category for arbitrary
actions of $G$ on $T$. So if $H$ acts freely and properly on $T$, we
can apply \cite[Proposition~4.1]{aHKRW-part2} with $G=H\times K$,
$N=H$ and $G/N=K$. This gives a fixed-point functor $\Fix_H^K$ from
$\Aa(H,K,(C_0(T),\rt))$ to $\Aa(K,(C_0(T/H),\rt))$ with object and
morphism maps
\[
(A,\sigma\times\tau,\phi)\mapsto
(A^\sigma,(\sigma\times\tau)^K,\phi_H) \quad\text{and}\quad [X,u\times
v] \mapsto [\Fix(X,u),(u\times v)^K].
\]
Now we introduce the notation
\[
\bar\tau = (\sigma\times\tau)^{K}, \quad X^u = \Fix(X,u),
\quad\text{and}\quad \bar v = (u\times v)^K.
\]
and define $\Fix_H$ to be the composition of $\Fix_H^K$ with the
forgetful functor from $\Aa(K,(C_0(T/H),\rt))$ to $\Aa(K)$. Then we
have:

\begin{prop}
  \label{prop-fixh-funct}
  Suppose that the action of $H$ on $T$ is free and proper.  Then the
  assignments
  \begin{equation}
    \label{eq:1}
    (A,\sigma,\tau,\phi)\mapsto (A^{\sigma},\bar\tau)
    \quad\text{and}\quad 
    [X,u,v]\mapsto [X^{u},\bar v]
  \end{equation}
  define a functor $\Fix_{H}$ from $\Aa(H,K,(C_{0}(T),\rt))$ to
  $\Aa(K)$.
\end{prop}

\section{Naturality}
\label{sec:sit}

Suppose that the actions of $H$ and $K$ on $T$ are free and proper.
Then for each object $(A,\sigma,\tau,\phi)$ of
$\Aa(H,K,(C_0(T),\rt))$, the hypotheses of
\cite[Theorem~4.4]{aHRWproper2} are satisfied; therefore $A_0 :=
\spn\{ \phi(f)a\phi(g) \mid f,g\in C_c(T), a\in A\}$ can be completed
to give an
$A^\sigma\rtimes_{\bar\tau,r}K$\;--\;$A^\tau\rtimes_{\bar\sigma,r}H$
imprimitivity bimodule, which we will denote here by
$X(A,\sigma,\tau,\phi)$.  The isomorphism class
$[X(A,\sigma,\tau,\phi)]$ is an isomorphism in the category $\CC$ from
$A^\sigma\rtimes_{\bar\tau,r}K=\RCP\circ\Fix_{H}(A,\sigma,\tau,\phi)$
to
$A^\tau\rtimes_{\bar\sigma,r}H=\RCP\circ\Fix_{K}(A,\sigma,\tau,\phi)$.

\begin{thm}
  \label{thm-sym-main}
  Suppose that $H$ and $K$ act freely and properly on $T$.  Then
  \begin{equation*}
    (A,\sigma,\tau,\phi)\mapsto [X(A,\sigma,\tau,\phi)]
  \end{equation*}
  is a natural isomorphism between the functors $\RCP\circ\Fix_{H}$
  and $\RCP\circ\Fix_{K}$ from $\Aa(H,K,(C_0(T),\rt))$ to $\CC$.
\end{thm}

Our strategy for proving Theorem~\ref{thm-sym-main} is as follows.
First, we present a natural isomorphism between the functors $\Fix_H$
and $\RCP_H$ defined in Section~\ref{sec:finding-functors}.  This is
done in Proposition~\ref{prop-equiv-main}, which is an equivariant
version of \cite[Theorem~3.5]{aHKRW-trans}.  Composing with $\RCP$
gives a natural isomorphism between $\RCP\circ\Fix_H$ and
$\RCP\circ\RCP_H$ (Corollary~\ref{cor-one-sided-equiv-main}).  Since
the iterated crossed-product functors $\RCP\circ\RCP_H$ and
$\RCP\circ\RCP_K$ are easily seen to be naturally isomorphic, we
obtain a natural isomorphism by composition:
\[
\RCP\circ\Fix_H \sim \RCP\circ\RCP_H \sim \RCP\circ\RCP_K \sim
\RCP\circ\Fix_K.
\]
The last step of the proof is to identify the bimodules underlying
this composition with $X(A,\sigma,\tau,\phi)$; this requires the
reworking of \cite[Theorem~4.4]{aHRWproper2} in
Proposition~\ref{prop-4.4redeux}.

For each object $(A,\sigma,\tau,\phi)$ of $\Aa(H,K,(C_0(T),\rt))$, the
triple $(A,\sigma,\phi)$ is an object of $\Aa(H,(C_0(T),\rt))$ with
$H$ acting freely and properly. Thus Rieffel's theory provides an
$A\rtimes_{\sigma,r}H$\;--\;$A^\sigma$ imprimitivity bimodule
$Z(A,\sigma,\phi)$ which is a completion of $A_0:=C_{c}(T)AC_{c}(T)$.
Now $\tau$ restricts to an action of $K$ on $A_0$ which, by the proof
of \cite[Proposition~4.2]{aHRWproper2}, extends to an action
$(\tau\rtimes\id_K,\tau,\bar\tau)$ of $K$ on $Z(A,\sigma,\phi)$.  Thus
$[Z(A,\sigma,\phi),\tau]$ is an isomorphism in the category $\Aa(K)$
between $(A\rtimes_{\sigma,r}H,\tau\rtimes\id_K) =
\RCP_H(A,\sigma,\tau,\phi)$ and $(A^\sigma,\bar\tau) =
\Fix_H(A,\sigma,\tau,\phi)$.

\begin{prop}
  \label{prop-equiv-main}
  Suppose that the action of $H$ on $T$is free and proper. Then
  \begin{equation*}
    (A,\sigma,\tau,\phi)\mapsto [Z(A,\sigma,\phi),\tau]
  \end{equation*}
  is a natural isomorphism between the functors $\RCP_{H}$ and
  $\Fix_{H}$ from the category $\Aa(H,K,(C_0(T),\rt))$ to $\Aa(K)$.
\end{prop}

\begin{proof}
  We follow the ``canonical decomposition'' strategy of the proof
  of~\cite[Theorem~3.5]{aHKRW-trans}.  So suppose $[X,u,v]$ is a
  morphism from $(A,\sigma,\tau,\phi)$ to $(B,\mu,\nu,\psi)$ in
  $\Aa(H,K,(C_0(T),\rt))$.  Then
  by~\cite[Corollary~2.3]{aHKRW-trans},\footnote{As we have observed
    earlier, we can apply \cite[Corollary~2.3]{aHKRW-trans} to
    $\Aa(H\times K,(C_{0}(T),\rt))$ without the assumption that
    $H\times K$ acts freely and properly on $T$.}  there exists an
  isomorphism $[Y,u,v]: (\K,\zeta,\eta,\chi) \to (B,\mu,\nu,\psi)$ in
  $\Aa(H,K,(C_0(T),\rt))$ given by \hbox{$\K$--$B$} imprimitivity
  bimodule $Y$, and a morphism $[\kappa]: (A,\sigma,\tau,\phi)\to
  (\K,\zeta,\eta,\chi)$ in $\Aa(H,K,(C_0(T),\rt))$ coming from a
  $\sigma$\;--\;$\zeta$ and $\tau$\;--\;$\eta$ equivariant
  nondegenerate homomorphism $\kappa: A\to M(\K)$, such that $[X,u,v]$
  is the composition:
  \begin{equation*}
    \xymatrix{
      (A,\sigma,\tau,\phi)
      \ar[rr]^{\kappa}
      &&(\K,\zeta,\eta,\chi)
      \ar[rr]^{(Y,u,v)}
      &&(B,\mu,\nu,\psi).
    }
  \end{equation*}

  Consider the following diagram in $\Aa(K)$, which is an equivariant
  version of diagram~(4.9) from the proof of
  \cite[Theorem~3.5]{aHKRW-trans}:
  \begin{equation}\label{factoring}\begin{split}
      \xymatrix@R+1pc@C-0.7pc{(A\rtimes_{\sigma,r}
        H,\tau\rtimes\id_{H}) \ar[dr]^-{\kappa\rtimes_r H}
        \ar[rrr]^-{(Z(A,\sigma,\phi),\tau)} \ar[dd]_{(X\rtimes_{u,r}
          H, v\rtimes\id_{H})} &&& (A^\sigma,\bar{\tau})
        \ar[dd]^{\hbox to0pt{$(\scriptstyle X^u,\bar v)$\hss}}
        \ar[dl]_{\kappa|}\\
        &**[l](\mathcal{K}\rtimes_{\zeta,r}
        H,\eta\rtimes\id_{H})\ar[r]^-{(Z(\K,\zeta,\chi),\eta)}
        \ar[dl]^{\qquad(Y\rtimes_{u,r} H,
          v\rtimes\id_{H})}&(\mathcal{K}^{\zeta},\bar\eta)\ar[dr]_{(Y^{u},\bar
          v)}\\
        (B\rtimes_{\mu,r} H,\nu\rtimes\id_{H})
        \ar[rrr]_-{(Z(B,\mu,\psi),\nu)} &&&(B^{\mu},\bar{\nu}).}
    \end{split}
  \end{equation}
  Here the arrow labeled $\kappa\rtimes_r H$ denotes (the isomorphism
  class of) $\K\rtimes_{\zeta,r}H$ viewed as an equivariant
  right-Hilbert
  $(A\rtimes_{\sigma,r}H)$--$(\K\rtimes_{\zeta,r}H)$-bimodule with the
  left action given by $\kappa\rtimes_r H$, and similarly for
  $\kappa|$.  Thus the left and right triangles of~\eqref{factoring}
  commute by functoriality of $\RCP_H$ and $\Fix_H$.

  The upper quadrilateral of~\eqref{factoring} is an equivariant
  version of diagram~(3.1) in \cite[Theorem~3.2]{kqrproper} (see also
  \cite[Remark~3.3]{kqrproper}).  That result provides an isomorphism
  $\Phi:Z(A,\sigma,\tau)\tensor_{A^\sigma}\K^\zeta\to
  Z(\K,\zeta,\chi)$ such that $\Phi(a\tensor
  E^{H}(c))=\kappa(a)E^{H}(c)$ for $a\in A_0$ and $c\in \K_0 = \spn\{
  \chi(f) d \chi(g) \mid f,g\in C_c(T), d\in \K \}$, where $E^H$ is
  the averaging process which maps $\K_0$ into $\K^\zeta$ (see
  \cite[Section~2]{kqrproper}).  Hence to prove that this upper
  quadrilateral commutes, we need to check that
  \begin{equation}
    \label{eq:4}
    \Phi\circ(\tau\otimes\bar\eta)_k=\eta_k\circ \Phi
    \qquad\text{for $k\in K$.}
  \end{equation}
  We break off our argument for a lemma:\

  \begin{lemma}
    \label{lem-eh-tech}
    For $k\in K$ and $c\in \K_{0}$,
    $\bar\eta_k\bigl(E^{H}(c)\bigr)=E^{H}\bigl(\eta(c)\bigr)$.
  \end{lemma}
  \begin{proof}
    By linearity, we may suppose $c=fdg$ with $f,g\in C_{c}(T)$ and
    $d\in\K$, where to reduce clutter we write $fdg$ for
    $\chi(f)d\chi(g)$.  Then for $h\in C_{c}(T)$ and $k\in K$, we
    compute, using \cite[Lemma~2.2]{kqrproper},
    \begin{align*}
      h\bar{\eta}_{k}\bigl(E^{H}(fdg)\bigr) &= \bar{\eta}_{k} \bigl(
      (\eta_{k})^{-1}(h) E^{H}(fdg)\bigr) = \bar{\eta}_{k}
      \Bigl(\int_{H}(\eta_{k})^{-1}(h)\zeta_{s}(fdg)\,ds\Bigr).
    \end{align*}
    Thus, since the integral is norm-convergent and since $\eta$ and
    $\zeta$ commute, we have
    \[
    h\bar{\eta}_{k}\bigl(E^{H}(fdg)\bigr)=\int_{H}
    h\zeta_{s}\bigl(\eta_{k}(fdg)\bigr)\,ds= h
    E^{H}\bigl(\eta_{k}(fdg)\bigr).\qedhere
    \]
  \end{proof}

\noindent\textit{End of the proof of Proposition~\ref{prop-equiv-main}.}
For $k\in K$, $a\in A_0$ and $c\in \K_0$, Lemma~\ref{lem-eh-tech}
gives
\begin{align*}
  \Phi\bigl(\tau_k\tensor \bar{\eta}_{k}(a\tensor E^{H}(c))\bigr)
  &=\Phi\bigl(\tau_k(a)\tensor\bar{\eta}_{k} \bigl(E^{H}(c)\bigr)\bigr)\\
  &=\Phi\bigl(\tau_k(a)\tensor E^{H}\bigl(\eta_k(c)\bigr)\bigr)\\
  &=\kappa\bigl(\tau_k(a)\bigr)E^{H}\bigl(\eta_k(c)\bigr)\\
  &=\eta_k\bigl(\kappa(a)\bigr)\bar\eta_k\bigl(E^{H}(c)\bigr)\\
  &=\eta_k\bigl(\kappa(a)E^{H}(c)\bigr) \\
  &=\eta_k\bigl(\Phi(a\tensor E^{H}(c)\bigr).
\end{align*}
This establishes \eqref{eq:4}, and shows that the upper quadrilateral
of \eqref{factoring} commutes.

Now we turn to the bottom quadrilateral of~\eqref{factoring}.  Here,
all the morphisms arise from imprimitivity bimodules, and we showed in
the proof of \cite[Theorem~3.5]{aHKRW-trans} that such a diagram
commutes in $\CC$ (that is, without the actions of~$K$) by observing
that $L(Y)^{L(u)}=L(Y^{u})$ and that $L(Y)\rtimes_{L(u),r}H$ is
isomorphic to $L(Y\rtimes_{u,r}H)$ by
\cite[Proposition~3.4]{aHRWproper}.  Then we applied
\cite[Lemma~4.6]{enchilada} to get the requisite
imprimitivity-bimodule isomorphisms.  To see that the diagram commutes
in $\Aa(K)$ just requires that these isomorphisms be equivariant, and
this is proved in \cite[Lemma~4.9]{enchilada}. This completes the
proof of Proposition~\ref{prop-equiv-main}.
\end{proof}

\begin{cor}
  \label{cor-one-sided-equiv-main}
  Suppose $H$ and $K$ act freely and properly.  Then the assignment
  \begin{equation*}
    (A,\sigma,\tau,\phi)\mapsto [Z(A,\sigma,\phi)\rtimes_{\tau,r}K]
  \end{equation*}
  is a natural isomorphism between the functors $\RCP\circ\RCP_{H}$
  and $\RCP\circ\Fix_{H}$ from the category $\Aa(H,K,(C_0(T),\rt))$ to
  $\CC$.
\end{cor}

The corollary, which can be viewed as an asymmetric version of
Theorem~\ref{thm-sym-main}, is an immediate consequence of
Proposition~\ref{prop-equiv-main} and the following elementary lemma.

\begin{lemma}
  \label{lem-funct-nonsense}
  Suppose that the assigment $A\mapsto\zeta_A$ is a natural
  isomorphism between functors $F_{1}$ and $F_{2}$ into a category
  $\texttt{C}$.  If $F_{3}$ is a functor defined on $\texttt{C}$, then
  the assigment $A\mapsto F_{3}(\zeta_{A})$ is a natural isomorphism
  between $F_{3}\circ F_{1}$ and $F_{3}\circ F_{2}$.
\end{lemma}

Given an imprimitivity bimodule $Y$, we denote the dual bimodule by
\[
Y^\sim = \{ \flat(y) \mid y\in Y \}.
\]

\begin{prop}
  \label{prop-4.4redeux}
  Suppose $H$ and $K$ act freely and properly.  For each object
  $(A,\sigma,\tau,\phi)$ of $\Aa(H,K,(C_{0}(T),\rt))$, there is an
  $(A^\sigma\rtimes_{\bar\tau,r}K)$--$(A^\tau\rtimes_{\bar\sigma,r}H)$
  imprimitivity bimodule isomorphism
  \begin{equation*}
    X(A,\sigma,\tau,\phi)\cong
    \bigl(Z(A,\sigma,\phi)\rtimes_{\tau,r}K\bigr)^{\sim}\tensor_{\Sigma}
    \bigl(Z(A,\tau,\phi)\rtimes_{\sigma,r}H\bigr),
  \end{equation*}
  where the balanced tensor product incorporates the natural
  isomorphism
  \begin{equation}\label{rug}
    \Sigma:  (A\rtimes_{\sigma,r}H)\rtimes_{\tau\rtimes\id,r}K
    \cong (A\rtimes_{\tau,r}K)\rtimes_{\sigma\rtimes\id,r}H.
  \end{equation}
\end{prop}
\begin{proof}
  Theorem 4.4 of \cite{aHRWproper2} implies that
  $X(A,\sigma,\tau,\phi)$ is isomorphic to the tensor product
  \begin{equation}\label{eq:10}
    (X\rtimes_{\tau,r}K)\tensor_{\Sigma} \bigl(
    Z(A,\tau,\phi)\rtimes_{\sigma,r}H\bigr), 
  \end{equation}
  where the unfortunately-named $X$ in \eqref{eq:10} is a module built
  on $A_{0}$ as defined in the beginning of Section~2 of
  \cite{aHRWproper2}, and is equivariantly isomorphic to
  $Z(A,\sigma,\phi)^\sim$.  The map on $C_c(K,A_0)$ given by $f\mapsto
  \flat(f^{*})$, where $f^{*}(k):=
  \Delta_K(k)^{-1}\tau_k\bigl(f(k^{-1})^{*}\bigr)$, extends to an
  imprimitivity bimodule isomorphism of $X\rtimes_{\tau,r}K$ onto
  $(Z(A,\sigma,\phi)\rtimes_{\tau,r}K)^\sim$, and the result follows.
\end{proof}

\begin{proof}[Proof of Theorem~\ref{thm-sym-main}]
  It follows from Corollary~\ref{cor-one-sided-equiv-main} that the
  assignment
  \[
  (A,\sigma,\tau,\phi) \mapsto
  \bigl[Z(A,\sigma,\phi)\rtimes_{\tau,r}K\bigr]^{-1}
  =\bigl[\bigl(Z(A,\sigma,\phi)\rtimes_{\tau,r}K\bigr)^\sim\bigr]
  \]
  is a natural isomorphism between $\RCP\circ\Fix_H$ and
  $\RCP\circ\RCP_H$; symmetrically,
  \[
  (A,\sigma,\tau,\phi) \mapsto [Z(A,\tau,\phi)\rtimes_{\sigma,r}H]
  \]
  is a natural isomorphism between $\RCP\circ\RCP_K$ and
  $\RCP\circ\Fix_K$.  It is a straightforward computation to verify
  that the isomorphism $\Sigma$ at~\eqref{rug} indeed gives a natural
  isomorphism between $\RCP\circ\RCP_H$ and $\RCP\circ\RCP_K$;
  composing these three and using Proposition~\ref{prop-4.4redeux}, we
  see that the assignment
  \begin{align*}
    (A,\sigma,\tau,\phi) &\mapsto [Z(A,\tau,\phi)\rtimes_{\sigma,r}H]
    \circ[\Sigma]
    \circ\bigl[\bigl(Z(A,\sigma,\phi)\rtimes_{\tau,r}K\bigr)^\sim\bigr]\\
    &=
    \bigl[\bigl(Z(A,\sigma,\phi)\rtimes_{\tau,r}K\bigr)^{\sim}\tensor_{\Sigma}
    \bigl(Z(A,\tau,\phi)\rtimes_{\sigma,r}H\bigr)\bigr]\\
    &= [X(A,\sigma,\tau,\phi)]
  \end{align*}
  is a natural isomorphism between $\RCP\circ\Fix_H$ and
  $\RCP\circ\Fix_K$, as desired.
\end{proof}

\section{Applications}
\label{sec:applications}
We view objects in $\Aa(H\times K)$ as triples $(B,\alpha,\beta)$
consisting of commuting actions $\alpha:H\to \Aut B$ and
$\beta:K\to\Aut B$.  Suppose that the actions of $H$ and $K$ on $T$
are free and proper. Then the symmetric imprimitivity theorem of
\cite{rae} says that $\Ind_H^T(B,\alpha)\rtimes_{\bbetatrt}K$ is
Morita equivalent to $\Ind_K^T(B,\beta)\rtimes_{\balphatrt}H$ (and
provides a concrete bimodule which is a completion of $C_c(T,B)$).  As
explained in \cite[Corollary~3]{aHRqs}, it follows from
\cite[Lemma~4.1]{qs} that the reduced crossed products are Morita
equivalent via a quotient $Y(B,\alpha,\beta)$ of the bimodule of
\cite{rae}.

\begin{cor}
  \label{cor-qs}
  Suppose that $H$ and $K$ act freely and properly on $T$.  Then the
  assignments
  \begin{equation*}
    (B,\alpha,\beta)\mapsto
    \Ind_H^T(B,\alpha)\rtimes_{\bbetatrt,r}K \quad\text{and}\quad
    (B,\alpha,\beta)\mapsto \Ind_K^T(B,\beta)\rtimes_{\balphatrt,r}H
  \end{equation*}
  are the object maps for functors $\FF_{H}$ and $\FF_{K}$ on
  $\Aa(H\times K)$, and the assignment
  \begin{equation*}
    (B,\alpha,\beta)\mapsto [Y(B,\alpha,\beta)]
  \end{equation*}
  is a natural isomorphism between $\FF_{H}$ and $\FF_{K}$.
\end{cor}

\begin{proof}
  Define $\phi:C_{0}(T)\to M(C_{0}(T,B))$ by $\phi(f)=f\tensor 1$.  It
  is not hard to see that there is a functor $\CT:\Aa(H\times K)\to
  \Aa(H,K,(C_{0}(T),\rt))$ which sends the object $(B,\alpha,\beta)$
  to $( C_{0}(T,B),\rt\otimes\alpha,\rt\otimes\beta,\phi)$.  Then it
  follows from Theorem~\ref{thm-sym-main} (and general nonsense such
  as Lemma~\ref{lem-funct-nonsense}) that the assignment
  \begin{equation*}
    (B,\alpha,\beta)\mapsto [X(
    C_{0}(T,B),\rt\otimes\alpha,\rt\otimes\beta,\phi)] 
  \end{equation*}
  is a natural isomorphism between $\RCP\circ \Fix_{H}\circ \CT$ and
  $\RCP\circ \Fix_{K}\circ \CT$.  Therefore we can complete the proof
  of the Corollary by showing that the relevant generalized
  fixed-point algebras coincide with the induced algebras, and that
  the modules $Y(B,\alpha,\beta)$ and $X(
  C_{0}(T,B),\rt\otimes\alpha,\rt\otimes\beta,\phi)$ are isomorphic as
  imprimitivity bimodules, hence define the same isomorphism in the
  category $\CC$.

  The dense subalgebra $A_{0}=C_{c}(T) C_{0}(T,B)C_{c}(T)$ of
  $C_0(T,B)$ is $C_{c}(T,B)$.  Then for $f,g\in A_{0}$, the
  $\Fix(C_{0}(T,B),\rt\otimes\alpha)$-valued inner product $\rip<f,g>$
  is multiplication by the function
  \begin{equation*}
    s\mapsto \int_{H}\alpha_{t}(f(st)^{*}g(st))\,dt,
  \end{equation*}
  which is the same as the $\Ind_H^T(B,\alpha)$-valued inner product
  on $C_{c}(T,B)$ in \cite{rae}.  Hence
  $\Fix(C_{0}(T,B),\rt\otimes\alpha)$ and $\Ind_H^T(B,\alpha)$ are the
  same subalgebra of $M( C_{0}(T,B))$.  A similar argument applies to
  $\Ind_K^T(B,\beta)$.

  To see that $Y(B,\alpha,\beta)$ is isomorphic to $X(
  C_{0}(T,B),\rt\otimes\alpha,\rt\otimes\beta, \phi)$, we just need to
  check that the two sets of actions and inner products on
  $C_{c}(T,B)$ coincide, and this is a straightforward calculation.
\end{proof}

When $K=\{e\}$, $\Ind_K^T(B,\beta)=C_{0}(T,B)$ and
\cite[Corollary~4]{aHRqs} implies that
\begin{equation*}
  C_{0}(T,B)\rtimes_{\rt\otimes\alpha}H=
  C_{0}(T,B)\rtimes_{\rt\otimes\alpha,r}H. 
\end{equation*}
Both the bimodule $X( C_{0}(T,B),\rt\otimes\alpha,\id,\phi)$ and the
bimodule $W(B,\alpha)$ in \cite{RW} are completions of $C_c(T,B)$, and
again the formulas for the actions and inner products turn out to be
the same. Hence we obtain the naturality of the Morita equivalence for
diagonal actions from \cite{RW}.
\begin{cor}
  \label{cor-rw}
  Suppose that $H$ acts freely and properly on $T$.  Then there are
  functors on $\Aa(H)$ whose maps on objects are given by
  \begin{equation*}
    (B,\alpha)\mapsto \Ind(C,\alpha)\quad\text{and}\quad (B,\alpha)
    \mapsto C_{0}(T,B)\rtimes_{\rt\otimes\alpha}H,
  \end{equation*}
  and the assignment
  \begin{equation*}
    (B,\alpha)\mapsto [W(B,\alpha)]
  \end{equation*}
  is a natural isomorphism between these two functors.
\end{cor}


\begin{thebibliography}{29}

\bibitem{CEOO} J. Chabert, S. Echterhoff and H. Oyono-Oyono,
  \emph{Shapiro's lemma for topological K-theory of groups},
  Comment. Math. Helv. \textbf{78} (2003), 203--225.

\bibitem{enchilada} S.~Echterhoff, S.~Kaliszewski, J.~Quigg and
  I.~Raeburn, \emph{A categorical approach to imprimitivity theorems
    for {$C\sp *$}-dynamical systems}, Mem.
  Amer. Math. Soc. \textbf{180} (2006), no.~850, viii+169 pages.


\bibitem{exel} R.~Exel, \emph{{Morita-Rieffel equivalence and spectral theory for integrable automorphism groups of $C^*$-algebras}}, J. Funct. Anal. \textbf{172} (2000), 404--465.

  
\bibitem{G} P. Green, \emph{$C^*$-algebras of transformation groups
    with smooth orbit space}, Pacific J. Math.  \textbf{72} (1977),
  71--97.

\bibitem{aHRqs} A. an~Huef and I. Raeburn, \emph{Regularity of induced
    representations and a theorem of Quigg and Spielberg},
  Math. Proc. Camb. Phil. Soc. \textbf{133} (2002), 149--159.


\bibitem{aHKRW-trans} A.~an Huef, S. Kaliszewski, I. Raeburn and
  D.P. Williams, \emph{Naturality of Rieffel's {M}orita equivalence
    for proper actions}, to appear in Algebr. Represent.
  Theory. (arXiv:0810.2819)

\bibitem{aHKRW-part2} A.~an Huef, S. Kaliszewski, I. Raeburn and
  D.P. Williams, \emph{Fixed-point algebras for proper actions and
    crossed products by homogeneous spaces}, to appear in Illinois
  J. Math. (arXiv:0907.0681)

\bibitem{aHRWproper} A.~an~Huef, I.~Raeburn and D.P.  Williams,
  \emph{Proper actions on imprimitivity bimodules and decompositions
    of Morita equivalences}, J. Funct. Anal. \textbf{200} (2003),
  401--428.

\bibitem{aHRWproper2} A.~an~Huef, I.~Raeburn and D.P.  Williams,
  \emph{A symmetric imprimitivity theorem for commuting proper
    actions}, Canad. J. Math. \textbf{57} (2005), 983--1011.
  

\bibitem{kqrproper} S. Kaliszewski, J. Quigg and I. Raeburn,
  \emph{Proper actions, fixed-point algebras and naturality in
    nonabelian duality}, J. Funct. Anal. \textbf{254} (2008),
  2949--2968.
  
  
\bibitem{kas} G.G. Kasparov, \emph{Equivariant {$KK$}-theory and the
    Novikov conjecture}, Invent. Math. \textbf{91} (1988), 147--201.
 
\bibitem{KW} E. Kirchberg and S. Wassermann, \emph{Permanence
    properties of $C^*$-exact groups}, Doc. Math.  \textbf{4} (1999),
  513--558.


 
\bibitem{mrw} P.S.  Muhly, J. Renault and D.P.  Williams,
  \emph{Equivalence and isomorphism for groupoid $C^*$-algebras},
  J. Operator Theory \textbf{17} (1987), 3--22.

\bibitem{MW} P.S.  Muhly and D.P. Williams, \emph{Equivalence and
    disintegration theorems for Fell bundles and their
    $C^*$-algebras}, Dissertationes Math. \textbf{56} (2008), 1--57.

\bibitem{pr} D.~Pask and I.~Raeburn, \emph{Symmetric imprimitivity
    theorems for graph $C^*$-algebras}, Internat. J. Math. \textbf{12}
  (2001), 609--623.

\bibitem{qs} J.C. Quigg and J. Spielberg, \emph{Regularity and
    hyporegularity in $C^*$-dynamical systems}, Houston
  J. Math. \textbf{18} (1992), 139--152.
  
\bibitem{rae} I. Raeburn, \emph{Induced $C^*$-algebras and a symmetric
    imprimitivity theorem}, Math. Ann. \textbf{280} (1988), 369--387.


\bibitem{RW} I.~Raeburn and D.P. Williams, \emph{Pull-backs of
    $C^*$-algebras and crossed products by certain diagonal actions},
  Trans. Amer. Math. Soc. \textbf{287} (1985), 755--777.
  
\bibitem{ren} J. Renault, \emph{Repr\'esentation des produits
    crois\'es d'alg\`ebres de groupo\"ides}, J. Operator Theory
  \textbf{18} (1987), 67--97.

\bibitem{rie:pspm82} M.A. Rieffel, \emph{Applications of strong
    {M}orita equivalence to transformation group $C^*$-algebras},
  Operator Algebras and Applications, Proc. Symp. Pure Math., vol.~38,
  {P}art~{I}, Amer.  Math. Soc., Providence, 1982, pages~299--310.

 
\bibitem{proper} M.A. Rieffel, \emph{Proper actions of groups on
    {$C\sp *$}-algebras}, Mappings of Operator Algebras, Progress in
  Math., vol.~84, Birkh\"auser, Boston, 1990, pages~141--182.

\bibitem{integrable} M.A. Rieffel, \emph{Integrable and proper actions
    on {$C\sp *$}-algebras, and square-integrable representations of
    groups}, Expositiones Math. \textbf{22} (2004), 1--53.


\end{thebibliography}
\end{document}